\documentclass[12pt]{amsart}
\usepackage{amscd}
\usepackage{amsfonts}
\usepackage{amssymb,latexsym}
\usepackage{enumerate}
\usepackage{amsmath}
\usepackage{mathrsfs}
\usepackage{amsthm}
\usepackage{hyperref}
\usepackage[square, comma, sort&compress, numbers]{natbib}
\usepackage{bm}
\usepackage{esint}
\usepackage{fancyhdr}
\usepackage{lastpage}
\usepackage{layout}
\usepackage{ulem}
\usepackage{extarrows}
\usepackage{geometry}
\newcommand{\ol}{\overline}

\newcommand{\pa}{\partial}

\DeclareMathOperator\tr{tr}
\DeclareMathOperator\id{id}
\DeclareMathOperator\rank{rank}
\DeclareMathOperator\Ric{Ric}
\DeclareMathOperator\Sca{Sca}
\DeclareMathOperator\dist{dist}
\DeclareMathOperator\vol{vol}
\DeclareMathOperator\dvol{dvol}
\DeclareMathOperator\End{End}

\DeclareMathOperator\Div{div}

\DeclareMathOperator\GL{GL}
\DeclareMathOperator\SL{SL}
\makeatletter
\begin{document}
\newcounter{remark}
\newcounter{theor}
\setcounter{remark}{0}
\setcounter{theor}{1}
\newtheorem{claim}{Claim}[section]
\newtheorem{theorem}{Theorem}[section]
\newtheorem{lemma}{Lemma}[section]
\newtheorem{corollary}{Corollary}[section]
\newtheorem{corollarys}{Corollary}
\newtheorem{proposition}{Proposition}[section]
\newtheorem{question}{Question}[section]
\newtheorem{defn}{Definition}[section]
\newtheorem{examp}{Example}[section]
\newtheorem{assumption}{Assumption}[section]
\newtheorem{rem}{Remark}[section]
\newtheorem*{theorem1}{Theorem}
\newtheorem*{maintheorem}{Main Theorem}
\newtheorem*{maintheorem2}{Main Theorem \uppercase\expandafter{\romannumeral2}}
\numberwithin{equation}{section}
\title[Liouville theorems for harmonic 1-forms on gradient Ricci solitons]
{Liouville theorems for harmonic 1-forms \\ on gradient Ricci solitons}
\author{Chenghong He}
\address{Chenghong He, School of Mathematics and Statistics, Nanjing University of Science and Technology, Nanjing 210094, People's Republic of China}
\email{Hech@njust.edu.cn}
\author{Di Wu}
\address{Di Wu, School of Mathematics and Statistics, Nanjing University of Science and Technology, Nanjing 210094, People's Republic of China}
\email{wudi@njust.edu.cn}
\author{Xi Zhang}
\address{Xi Zhang, School of Mathematics and Statistics, Nanjing University of Science and Technology, Nanjing 210094, People's Republic of China}
\email{mathzx@njust.edu.cn}
\thanks{The research is supported by the National Key R\&D Program of China 2020YFA0713100. The authors are partially supported by the National Natural Science Foundation of China No.12141104 and No.12431004. Di Wu is also supported by the project funded by China Postdoctoral Science Foundation 2023M731699, the Jiangsu Funding Program for Excellent Postdoctoral Talent 2022ZB282, and the Natural Science Foundation of Jiangsu Province No.BK20230902.}
\thanks{Di Wu would like to thank Shiyu Zhang for several useful discussions.}
\maketitle
\begin{abstract}
We prove that there is no nontrivial $L^2$-integrable harmonic 1-form on noncompact complete gradient steady Ricci solitons or noncompact complete gradient shrinking K\"{a}hler-Ricci solitons. As an application, it can be used to distinguish certain flat vector bundles that arise from fundamental group representations into $\SL(r,\mathbb{C})$.
\end{abstract}
\section{Introduction}
As is well-known, the classical Liouville theorem for harmonic functions states that there is no nontrivial positive harmonic function on $\mathbb{R}^n$. Using the logarithmic gradient estimate, it is Yau \cite{Ya1975} who successfully extended this theorem to the context of noncompact complete manifolds with nonnegative Ricci curvatures nearly fifty years ago. Yau \cite{Ya1976} also proved that any nonnegative $L^p$-integrable sub-harmonic function on a complete manifold must be a constant when $p>1$, while Li-Schoen \cite{LS1984} obtained similar result on noncompact complete manifolds with nonnegative Ricci curvatures when $p\in(0,1]$. In particular, there is no nontrivial harmonic function with finite energy on a complete manifold with nonnegative Ricci curvature. Since then the subject has witnessed a tremendous growth, the reader may consult \cite{CW2019,Li2006} for the comprehensive exposition of Liouville properties for harmonic functions as well as related topics.
\par It is interesting to study various Liouville type properties for many other partial differential equations. Motivated by Liouville theorems for harmonic functions on gradient Ricci solitons studied by Munteanu-Se\v{s}um \cite{MS2013}, we want to investigate Liouville theorems for harmonic 1-forms on gradient Ricci solitons. A noncompact complete gradient Ricci soliton $(M,g,f)$ consists of a noncompact complete Riemannian manifold $(M,g)$, a smooth function $f$ (called the potential function), and a constant $\rho$ such that
\begin{equation}\begin{split}
\Ric_g+\nabla df=\rho g,
\end{split}\end{equation}
where we denote by $\nabla$ the Levi-Civita connection. After rescaling the metric we may assume that $\rho\in\{-\frac{1}{2},0,\frac{1}{2}\}$. It is called shrinking if $\rho=\frac{1}{2}$, steady if $\rho=0$ and expanding if $\rho=-\frac{1}{2}$. Clearly, a gradient Ricci soliton is simply an Einstein manifold if $f$ is a constant function. Gradient Ricci solitons are not only the generalization of Einstein manifolds but also arise often as singularity models of the Ricci flow and that is why underlying them is an important question in the field. 
\par For a gradient Ricci soliton $(M,g,f)$, the drifted Laplacian and Bakry-\'{E}mery Ricci curvature are given by
\begin{equation}\begin{split}
		\Delta_{g,f}=\Delta_g-<\nabla f,\bullet>,\ \Ric_{g,f}=\Ric_g+\nabla df.
\end{split}\end{equation}
Since by definition gradient steady Ricci solitons or gradient shrinking Ricci solitons have nonnegative Bakry-\'{E}mery Ricci curvatures, they sometimes behave like Riemannian manifolds with nonnegative Ricci curvatures. So it is quite natural to ask whether many known Liouville type properties on noncompact complete manifolds with nonnegative Ricci curvatures can be extended to the context of noncompact complete gradient steady Ricci solitons or noncompact complete gradient shrinking Ricci solitons.
\par There are many works studying various Liouville type properties for $\Delta_{g,f}$-harmonic objects on gradient steady or shrinking Ricci solitons (or more generally, smooth metric measure spaces with Bakry-\'{E}mery Ricci curvatures bounded from below), see \cite{Br2013,GZ2018,MW2011,Na2010,PRS2011,WW2009,WW2015} and references therein. It should be mentioned that the issue for standard harmonic objects sometimes seems more delicate and many results remain unclearly. In this direction, Munteanu-Sesum \cite{MS2013} obtained the corresponding result for harmonic functions with finite energies. The issue still deserves further research for harmonic functions with point-wise conditions, see \cite{MO2024,MW2014}.
\par The present paper investigates the aforementioned issue for harmonic 1-forms, where a differential form is called harmonic if it is both closed and co-closed. The main results are stated as follows and they generalize well-known results on noncompact complete manifolds with nonnegative Ricci curvatures, see \cite{GW1981,Ya1976} for example.
\begin{theorem}\label{1formsteady}
There is no nontrivial $L^2$-integrable harmonic 1-form on noncompact complete gradient steady Ricci solitions.
\end{theorem}
\begin{theorem}\label{1formshrink}
There is no nontrivial $L^2$-integrable harmonic 1-form on noncompact complete gradient shrinking K\"{a}hler-Ricci solitions.
\end{theorem}
It is known that the Liouville type property for harmonic functions can be used to investigate the geometry of complete manifolds at infinity. In this paper, we observe that the Liouviile type property of harmonic 1-forms can be used to study the fundamental group representations of flat vector bundles. In fact, Theorems \ref{1formsteady}, \ref{1formshrink} are applied to distinguish certain flat vector bundles that arise from fundamental group representations into $\SL(r,\mathbb{C})$, see Section 3. This is another motivation of our research. 
\section{Liouville theorems for harmonic 1-forms}
Let $(M,g,f)$ be a noncompact complete gradient Ricci soliton. In \cite{Ha1995}, Hamilton showed that the scalar curvature $\Sca_g$ satisfies
\begin{equation}\begin{split}\label{GRS}
		\Sca_g+|\nabla f|^2-2\rho f=\lambda,
\end{split}\end{equation}
for a constant $\lambda$. If $\rho=0$ (that is $\Ric_g=\nabla df$), since every steady soliton is in particular an ancient solution to the Ricci flow, we have (see \cite{Ca2010,Ch2009})
\begin{equation}\begin{split}\label{steady1}
		\Sca_g\geq0,\ \lambda\geq0,\ |\nabla f|\leq\sqrt{\lambda}.
\end{split}\end{equation}
On the other hand, Munteanu-Se\v{s}um \cite{MS2013} proved the volume growth
\begin{equation}\begin{split}\label{steady2}
		\frac{r}{c_1}\leq\vol_g(B_p(r))\leq c_1e^{\sqrt{r}},\ r>r_0,
\end{split}\end{equation}
for two constants $c_1$, $r_0$.
Next let us consider the case $\rho=\frac{1}{2}$ (that is $\Ric_g+\nabla df=\frac{1}{2}\rho g$), after suitable change of $f$ we may assume 
\begin{equation}\begin{split}\label{shrinking1}
		\Sca_g+|\nabla f|^2=f.
\end{split}\end{equation}
It is known from \cite{Ca2010,Ch2009} that $\Sca_g\geq0$. On the other hand, Cao-Zhou \cite{CZ2010} proved that for any fixed $p\in M$, we have
\begin{equation}\begin{split}\label{shrinking2}
		\frac{1}{4}(\dist(\bullet,p)-c_2)^2\leq f\leq\frac{1}{4}(\dist(\bullet,p)+c_2)^2,
\end{split}\end{equation}
for a constant $c_2$. In particular, it holds
\begin{equation}\begin{split}\label{shrinking3}
		|\nabla f|\leq\frac{1}{2}(\dist(\bullet,p)+c_2).
\end{split}\end{equation}
\par In the following, we let $\alpha$ be a $L^2$-integrable harmonic 1-form on $M$ and $\alpha^{\#}\in \Gamma(TM)$ be the dual of $\alpha$ via $g$. We also choose $\eta: M\rightarrow[0,1]$ to be a cut-off function such that $\eta=1$ on a $B_p(r)$, $\eta=0$ outside $B_p(2r)$ and $|\nabla\eta|\leq Cr^{-1}$ for a constant $C$. 
\begin{proof}[\textup{\textbf{Proof of Theorem \ref{1formsteady}}}]
Firstly, we apply the divergence theorem to get
\begin{equation}\begin{split}\label{Stokes1}
0&=\int_M\Div(\eta^2<df,\alpha>\alpha^{\#})\dvol_g
\\&=\int_M<df,\alpha><\nabla\eta^2,\alpha^{\#}>\dvol_g+\int_M\eta^2<<\nabla df,\alpha>,\alpha>\dvol_g
\\&+\int_M\eta^2<<df,\nabla\alpha>,\alpha>\dvol_g
+\int_M\eta^2<df,\alpha>\Div\alpha^{\#}\dvol_g.
\end{split}\end{equation}
Since $\alpha$ is harmonic, we have 
\begin{equation}\label{harmonicdivergence}
	\begin{split}
		\Div\alpha^{\#}=d^\ast\alpha=0.
	\end{split}
\end{equation}
Using the steady soliton equation $\Ric_g=\nabla df$, we get
\begin{equation}\label{RSS}
	\begin{split}
<<\nabla df,\alpha>,\alpha>
&=<<\Ric_g,\alpha>,\alpha>
=\Ric_g(\alpha^{\#},\alpha^{\#}).
	\end{split}
\end{equation}
For any two vector fields $X,Y$, we have $0=d\alpha(X,Y)
=\nabla_{X}\alpha(Y)
-\nabla_{Y}\alpha(X)$. Hence
\begin{equation}\label{symmetry2}
	\begin{split}
		<<df,\nabla\alpha>,\alpha>
		=<<\nabla\alpha,\alpha>,df>.
	\end{split}
\end{equation}
Using $(\ref{harmonicdivergence})$, $(\ref{RSS})$, $(\ref{symmetry2})$ on $(\ref{Stokes1})$, we arrive at
\begin{equation}\begin{split}\label{A1}
		0&=\int_M<df,\alpha><\nabla\eta^2,\alpha^{\#}>\dvol_g+\int_M\eta^2\Ric_g(\alpha^{\#},\alpha^{\#})\dvol_g
		\\&+\int_M\eta^2<<\nabla\alpha,\alpha>,df>\dvol_g.
\end{split}\end{equation}
\par Secondly, applying the divergence theorem again implies
\begin{equation}\begin{split}\label{A2}
		0&=\int_M\Div(\eta^2|\alpha|^2\nabla f)\dvol_g
		\\&=2\int_M\eta^2<<\nabla\alpha,\alpha>,df>\dvol_g+\int_M\eta^2|\alpha|^2\Delta_gf\dvol_g
		\\&+\int_M|\alpha|^2<\nabla\eta^2,\nabla f>\dvol_g.
\end{split}\end{equation}
Now using the Kato's inequality, we have 
\begin{equation}\begin{split}\label{A3}
\Delta_g|\alpha|^2
&=2\Ric_g(\alpha^{\#},\alpha^{\#})+2|\nabla\alpha|^2
\geq2\Ric_g(\alpha^{\#},\alpha^{\#})+2|\nabla|\alpha||^2.
\end{split}\end{equation}
\par Gathering $(\ref{A1})$, $(\ref{A2})$, $(\ref{A3})$, we deduce 
\begin{equation}\begin{split}
&2\int_M\eta^2|\nabla|\alpha||^2\dvol_g
\\&\leq-\int_M<\nabla\eta^2,\nabla|\alpha|^2>\dvol_g-2\int_M\eta^2\Ric_g(\alpha^{\#},\alpha^{\#})\dvol_g
\\&=-\int_M<\nabla\eta^2,\nabla|\alpha|^2>\dvol_g+2\int_M<df,\alpha><\nabla\eta^2,\alpha^{\#}>\dvol_g
\\&+2\int_M\eta^2<<\nabla\alpha,\alpha>,df>\dvol_g
\\&=-\int_M<\nabla\eta^2,\nabla|\alpha|^2>\dvol_g+2\int_M<df,\alpha><\nabla\eta^2,\alpha^{\#}>\dvol_g
\\&-\int_M\eta^2|\alpha|^2\Delta_gf\dvol_g-\int_M|\alpha|^2<\nabla\eta^2,\nabla f>\dvol_g
\\&\leq4\int_M\eta|\nabla\eta||\alpha||\nabla|\alpha||\dvol_g+C_1\int_M|\nabla\eta||\nabla f||\alpha|^2\dvol_g
\\&-\int_M\eta^2|\alpha|^2\Sca_g\dvol_g,
\end{split}\end{equation}
for a constant $C_1$. The Cauchy-Schwarz inequality gives
\begin{equation}\begin{split}
	4\int_M\eta|\nabla\eta||\alpha||\nabla|\alpha||\dvol_g
	&\leq\int_M\eta^2|\nabla|\alpha||^2\dvol_g+4\int_M|\nabla\eta|^2|\alpha|^2\dvol_g.
\end{split}\end{equation}
The above two inequalities imply
\begin{equation}\begin{split}
&\int_M\eta^2|\nabla|\alpha||^2\dvol_g+\int_M\eta^2|\alpha|^2\Sca_g\dvol_g
\\&\leq4\int_M|\nabla\eta|^2|\alpha|^2\dvol_g
+C_1\int_M|\nabla\eta||\nabla f||\alpha|^2\dvol_g
\\&\leq r^{-1}C(4+C_1\sqrt{\lambda})\int_{B_{p}(2r)\setminus B_p(r)}|\alpha|^2\dvol_g,
\end{split}\end{equation}
where we have used the uniform bound of $|\nabla f|$ (see $(\ref{steady1})$) and assumed $r$ large enough. By letting $r$ goes infinity and noting $\Sca_g\geq0$ (see $(\ref{steady1})$), we find $\nabla|\alpha|=0$ and thus $|\alpha|$ is constant. Since the volume is infinite (see $(\ref{steady2})$), we conclude $\alpha=0$ as $|\alpha|\in L^2$.
\end{proof}
\begin{proof}[\textup{\textbf{Proof of Theorem \ref{1formshrink}}}]
	Li's classical result in \cite{Li1990} asserts that any harmonic function on a noncompact complete K\"{a}hler manifold with suitable energy decay condition must be pluri-harmonic, it is a key point in obtaining the Liouville theorem for harmonic functions on noncompact complete gradient shrinking K\"{a}hler-Ricci solitions (see \cite{MS2013}). Along the way of our proof of Theorem \ref{1formshrink}, we need the following proposition.
\begin{proposition}\label{pluriharmonic}
	Let $\alpha$ be a harmonic 1-from on a noncompact complete K\"{a}hler manifold $(M,g)$ with $\int_{B_p(r)}|\alpha|^2\dvol_g=o(r^2)$ as $r$ goes infinity, then $\ol\pa\alpha^{1,0}=0$.
\end{proposition}
\begin{proof}
Set $n=\dim_{\mathbb{C}}M$ and $\omega_g$ to be the K\"{a}hler form, we have
\begin{equation}\begin{split}\label{33333}
		\int_M\eta^2\ol\pa\alpha^{1,0}\wedge\ol\pa\alpha^{1,0}\wedge\frac{\omega^{n-2}_g}{(n-2)!}
		&=\int_M\eta^2\ol\pa\left(\alpha^{1,0}\wedge\ol\pa\alpha^{1,0}\wedge\frac{\omega^{n-2}_g}{(n-2)!}\right)
		\\&=-\int_M2\eta\ol\pa\eta\wedge\alpha^{1,0}\wedge\ol\pa\alpha^{1,0}\wedge\frac{\omega^{n-2}_g}{(n-2)!}
		\\&\leq C_2\int_M|\eta||\nabla\eta||\alpha^{1,0}||\ol\pa\alpha^{1,0}|\dvol_g
		\\&\leq C_2\int_M(\epsilon\eta^2|\ol\pa\alpha^{1,0}|^2+\frac{1}{\epsilon}|\nabla\eta|^2|\alpha|^2)\dvol_g,
\end{split}\end{equation}
for a constant $C_2$ and any constant $\epsilon>0$. Note the harmoncity implies
 \begin{equation}\begin{split}\label{11111}
 		\pa\alpha^{0,1}+\ol\pa\alpha^{1,0}=0,\ \pa^\ast\alpha^{1,0}+\ol\pa^\ast\alpha^{0,1}=0.
 \end{split}\end{equation}
Set $\Lambda$ to be the formal adjoint operator of $\omega_g\wedge\bullet$ and locally, it holds $\Lambda\beta=-\sqrt{-1}g^{\ol{j}i}\beta_{i\ol{j}}$ for $\beta=\beta_{i\ol{j}}dz^i\wedge d\ol{z^j}$, where $(g^{\ol{j}i})_{n\times n}=(g_{i\ol{j}})_{n\times n}^{-1}$ and $g_{i\ol{j}}=g(dz^i,d\ol{z^j})$. Then we deduce
\begin{equation}\begin{split}\label{22222}
2\sqrt{-1}\Lambda\ol\pa\alpha^{1,0}
&=\sqrt{-1}\Lambda\ol\pa\alpha^{1,0}-\sqrt{-1}\Lambda\pa\alpha^{0,1}
\\&=\pa^\ast\alpha^{1,0}+\ol\pa^\ast\alpha^{0,1}
\\&=0.
\end{split}\end{equation}
Using $(\ref{11111})$ and $(\ref{22222})$, a direct computation yields
\begin{equation}\begin{split}\label{44444}
\int_M\eta^2\ol\pa\alpha^{1,0}\wedge\ol\pa\alpha^{1,0}\wedge\frac{\omega^{n-2}_g}{(n-2)!}
&=-\int_M\eta^2\ol\pa\alpha^{1,0}\wedge\pa\alpha^{0,1}\wedge\frac{\omega^{n-2}_g}{(n-2)!}
\\&=-\int_M\eta^2\ol\pa\alpha^{1,0}\wedge\ol{\ol\pa\alpha^{1,0}}\wedge\frac{\omega^{n-2}_g}{(n-2)!}
\\&=\int_M\eta^2(|\ol\pa\alpha^{1,0}|^2-|\Lambda\ol\pa\alpha^{1,0}|^2)\dvol_g
\\&=\int_M\eta^2|\ol\pa\alpha^{1,0}|^2\dvol_g.
\end{split}\end{equation}
Taking $\epsilon=(2C_2)^{-1}$, it follows from $(\ref{33333})$ and $(\ref{44444})$ that
\begin{equation}\begin{split}
		\int_M\eta^2|\ol\pa\alpha^{1,0}|^2\dvol_g
		&\leq4(C_2)^2\int_M|\nabla\eta|^2|\alpha|^2\dvol_g
		\\&\leq4r^{-2}(CC_2)^2\int_{B_{p}(2r)\setminus B_p(r)}|\alpha|^2\dvol_g,
\end{split}\end{equation}
By letting $r$ goes infinity, we conclude $\ol\pa\alpha^{1,0}=0$.
\end{proof}
By Proposition \ref{pluriharmonic}, it holds $\pa\alpha^{0,1}=0$, $\ol\pa\alpha^{1,0}=0$. 
From this we can obtain
\begin{equation}\begin{split}\label{pluriharmonicuse}
\nabla^{1,0}\alpha^{0,1}=0,\ \nabla^{0,1}\alpha^{1,0}=0,
\end{split}\end{equation}
The property of Ricci curvature of a K\"{a}hler manifold and $\Ric_g+\nabla df=\frac{1}{2}g$ imply 
\begin{equation}\begin{split}\label{Ricci}
\nabla^{1,0}\pa f=0,\ \nabla^{0,1}\ol\pa f=0.
\end{split}\end{equation}
We set $F=g(df,\alpha)$ and it can be written as
\begin{equation}\begin{split}
F&=g(\pa f,\alpha^{0,1})+g(\ol\pa f,\alpha^{1,0})
\triangleq F_1+F_2.
\end{split}\end{equation}
For $Z\in T^{1,0}M$, using $(\ref{pluriharmonicuse})$ and $(\ref{Ricci})$, we deduce
\begin{equation}\begin{split}
		\pa_{Z} F_1
		&=g(\nabla_{Z}\pa f,\alpha^{0,1})+g(\pa f,\nabla_Z\alpha^{0,1})
		=0,
\end{split}\end{equation}
\begin{equation}\begin{split}
		\ol\pa_{\ol{Z}} F_2
		&=g(\nabla_{\ol{Z}}\ol\pa f,\alpha^{1,0})+g(\ol\pa f,\nabla_{\ol{Z}}\alpha^{1,0})
		=0.
\end{split}\end{equation}
Therefore $\pa\ol\pa F=-\ol\pa\pa F_1+\pa\ol\pa F_2=0$ and it yields
\begin{equation}\begin{split}
\int_M|\nabla F|^2\eta^2\dvol_g
&=-\int_M\Delta_gFF\eta^2\dvol_g-\int_MF<\nabla F,\nabla\eta^2>\dvol_g
\\&=-\int_M2\sqrt{-1}\Lambda\pa\ol\pa FF\eta^2\dvol_g-\int_MF<\nabla F,\nabla\eta^2>\dvol_g
\\&\leq\int_M2\eta|F||\nabla F||\nabla\eta|\dvol_g
\\&\leq\frac{1}{2}\int_M|\nabla F|^2\eta^2\dvol_g+2\int_M|F|^2|\nabla\eta|^2\dvol_g.
\end{split}\end{equation}
From this and the fact that $|\nabla f|$ grows linearly (see $(\ref{shrinking3})$), we arrive at
\begin{equation}\begin{split}
\int_M|\nabla F|^2\eta^2\dvol_g
&\leq C_3\int_M|\nabla\eta|^2|\alpha|^2|\nabla f|^2\dvol_g
\\&\leq\frac{C_3C^2(\dist(\bullet,p)+c_2)^2}{4r^2}\int_{B_p(2r)\setminus B_p(r)}|\alpha|^2\dvol_g,
\end{split}\end{equation}
for a constant $C_3$. By letting $r$ goes infinity we conclude $\nabla F=0$ and hence $F$ is a constant. In fact, $(\ref{shrinking2})$ guarantees that it attains its minimum somewhere on
a compact subset of $M$, so $F=0$ somewhere and hence everywhere.
\par Now using $\Ric_{g,f}=\frac{1}{2}g$, $(\ref{symmetry2})$, $\nabla F=0$ and the Kato's inequality, we obtain
\begin{equation}\begin{split}\label{Bochner}
		\Delta_{g,f}|\alpha|^2
		&=2\Ric_g(\alpha^{\#},\alpha^{\#})+2|\nabla\alpha|^2-<\nabla f,\nabla|\alpha|^2>
		\\&=2\Ric_{g,f}(\alpha^{\#},\alpha^{\#})-2\nabla df(\alpha^{\#},\alpha^{\#})-2<df,<\nabla\alpha,\alpha>>+2|\nabla\alpha|^2
		\\&=|\alpha|^2
		-2<<\alpha,\nabla df>,\alpha>
		-2<df,<\nabla\alpha,\alpha>>+2|\nabla\alpha|^2
		\\&=|\alpha|^2
		-2<<\alpha,\nabla df>+<\nabla\alpha,df>,\alpha>+2|\nabla\alpha|^2
		\\&=|\alpha|^2-2<dF,\alpha>+2|\nabla\alpha|^2
		\\&\geq|\alpha|^2+2|\nabla|\alpha||^2.
\end{split}\end{equation}
On the other hand, we also have $\Delta_{g,f}|\alpha|^2=2|\alpha|\Delta_{g,f}|\alpha|+2|\nabla|\alpha||^2$. It follows
\begin{equation}\begin{split}
2\Delta_{g,f}|\alpha|\geq|\alpha|.
\end{split}\end{equation}
We find $\alpha=0$ by noting $\int_M|\alpha|^2e^{-f}\dvol_g\leq\int_M|\alpha|^2\dvol_g<\infty$ and applying Yau's Liouville theorem with respect to drifted Laplacian (see \cite{GZ2018,PRS2011,Ya1976}).
\end{proof}
\section{$\SL(r,\mathbb{C})$-flat vector bundles}
Let $(E,\nabla)$ be a complex flat vector bundle on a Riemannian manifold $(M,g)$, where $\nabla$ is a connection on $E$. Motivated by calculating the Euler characteristic number via the Gauss-Bonnet-Chern formula and its ramifications, one may ask: \textit{How to find metric compatible connections on $E$?} We may write for a Hermitian metric $H$ on $E$,
\begin{equation}\begin{split}
		\nabla=\nabla_H+\psi_{H},
\end{split}\end{equation}
where $\nabla_H$ is a connection preserving $H$ and $\psi_{H}\in A^1(\End(E))$. We define the energy
\begin{equation}\begin{split}
		\mathcal{E}_{\nabla}(H)=\frac{1}{2}\int_M|\psi_{H}|^2\dvol_g,
\end{split}\end{equation}
and the point is to minimize $\mathcal{E}_{\nabla}(H)$ when $H$ varies.	The critical point $H$ of $\mathcal{E}_{\nabla}$ is called a harmonic metric and it satisfies 
	\begin{equation}\begin{split}
			\nabla^\ast_H\psi_H=0,
	\end{split}\end{equation}
	where $\nabla^\ast_H$ denotes the formal adjoint operator of $\nabla_H$. It is known that harmonic metrics are key tools in the study of the so-called nonabelian Hodge theory in K\"{a}hler geometry, see \cite{Co1988,Do1987b,Hi1987,Si1988,Si1992}. Detecting harmonic metrics is a nonlinear system generalization of solving the Laplace equation and obstruction emerges. In fact, there is a Riemannian Kobayashi-Hitchin correspondence for harmonic metrics analogous to the Kobayashi-Hitchin correspondence \cite{Do1985,Do1987a,UY1986} in complex geometry, which was well established by Corlette \cite{Co1988} and Donaldson \cite{Do1987b} for flat vector bundles in the late 1980s, by Wu-Zhang \cite{WZ2023b} for arbitrary vector bundles recently. 
\par Consider the determinant line bundle $\det E$ endowed with the induced flat connection $\nabla^{\det E}$, whose connection form is given by the trace of that of $\nabla$. With respect to the induced metric $\det H$, we split it as 
\begin{equation}\begin{split}
		\nabla^{\det E}=\nabla^{\det E}_{\det H}+\psi^{\det E}_{\det H},
\end{split}\end{equation}
and we denote by $H_0$ the trivial Hermitian metric on $\det E$. Let $p\in M$ and $\pi_1(M)$ be the fundamental group with base point $p$. Since $F_\nabla=0$, the parallel transport along a closed curve starting at $p$ depends only on its homotopy class and it induces $\rho: \pi_1(M)\rightarrow\GL(r,\mathbb{C})$, where $r=\rank(E)$. Conversely, given $\rho: \pi_1(M)\rightarrow\GL(r,\mathbb{C})$, it corresponds to $E=\tilde{M}\times_\rho\mathbb{C}^r$ given by the universal covering space $\tilde{M}$ and the action of $[c]\in\pi_1(M)$ on $(q,v)\in\tilde{M}\times\mathbb{C}^r$ via $[c]\cdot(q,v)=([c]\cdot q,\rho([c])(v))$. Then we say that the flat vector bundle $E$ arises from the fundamental group representation $\rho$. 
\par It might be interesting to figure out more information on the representation $\rho$. On a compact Riemannian manifold with nonnegative Ricci curvature, the representation $\rho$ relative to any semi-simple complex flat vector bundle has image into $U(r)$, see \cite{WZ2023a} which also considered the noncompact case. On the other hand, recall the famous Chern conjecture predicts that any compact affine manifold has vanishing Euler characteristic. Thus the Chern conjecture holds for a compact affine manifold whose tangent bundle is semi-simple and which admits a Riemannian metric with nonnegative Ricci curvature.

As an application of the Liouviile type property in Theorems \ref{1formsteady}, \ref{1formshrink}, we are now in the position to prove the following theorem, which shows $\rho$ has image into $\SL(r,\mathbb{C})$.
\begin{theorem}\label{SL}
Let $(E,\nabla)$ be a complex flat vector bundle of rank $r$ on a noncompact complete gradient steady Ricci soliton or a noncompact complete gradient shrinking K\"{a}hler-Ricci soliton. Suppose $H$ is harmonic metric with finite energy and $\det H=H_0$, then $(E,\nabla)$ must arise from a fundamental group representation into $\SL(r,\mathbb{C})$.
\end{theorem}
\begin{proof}
Decomposing $F_\nabla$ into Hermitian part and anti-Hermitian part, we have 
	\begin{equation}\begin{split}\label{flatness1}
			D_H\psi_H=0,\ F_{\nabla_H}+\psi_H\wedge\psi_H=0,
	\end{split}\end{equation}
where $D_H$ is the exterior differential operator relative to $\nabla_H$. Together with the harmonicity of $H$, we find
\begin{equation}\begin{split}
			d\tr\psi_H=\tr D_H\psi_H=0,\ d^\ast\tr\psi_H=\tr\nabla_H^\ast\psi_H=0.
\end{split}\end{equation}
In addition, $\tr\psi_H$ is real since $\psi_H$ is self-adjoint. Applying Theorems \ref{1formsteady}, \ref{1formshrink}, we obtain $
\tr\psi_H=0$. If $(E,\nabla)$ arises form the representation $\rho$, then the induced flat line bundle $(\det E,\nabla^{\det E})$ arises from $\det\rho$. Since $\det H=H_0$ and 
\begin{equation}\begin{split}
\psi^{\det E}_{\det H}=\tr\psi_H=0,
\end{split}\end{equation}
it follows $\nabla^{\det E}=d$. Hence $(\det E,\nabla^{\det E})$ is trivial and thus $\det\rho=\id$.
\end{proof}


\begin{thebibliography}{1}
\footnotesize
\bibitem[Br]{Br2013}K. Brighton, A Liouville-type theorem for smooth metric measure spaces. J. Geom. Anal. 23 (2013), no. 2, 562-570.
\bibitem[Ca]{Ca2010}H-D. Cao, Recent progress on Ricci solitons. Recent advances in geometric analysis, 1–38, Adv. Lect. Math. (ALM), 11, Int. Press, Somerville, MA, 2010. 
\bibitem[CZ]{CZ2010}H-D. Cao and D. Zhou, On complete gradient shrinking Ricci solitons. J. Differential Geom. 85 (2010), no. 2, 175-185. 
\bibitem[Ch]{Ch2009}B-L. Chen, Strong uniqueness of the Ricci flow. J. Differential Geom. 82 (2009), no. 2, 363-382.
\bibitem[CW]{CW2019}T.H. Colding and W.P. Minicozzi. \uppercase\expandafter{\romannumeral2}, Liouville properties. ICCM Not. 7 (2019), no. 1, 16–26.
\bibitem[Co]{Co1988}K. Corlette, Flat $G$-bundles with canonical metrics. J. Differential Geom. 28 (1988), no. 3, 361-382.
\bibitem[Do1]{Do1985}S.K. Donaldson, Anti self-dual Yang-Mills connections over complex algebraic surfaces and stable vector bundles. Proc. London Math. Soc. (3). 50 (1985), no. 1, 1-26.
\bibitem[Do2]{Do1987a}S.K. Donaldson, Infinite determinants, stable bundles and curvature. Duke Math. J. 54 (1987), no. 1, 231-247.
\bibitem[Do3]{Do1987b}S.K. Donaldson, Twisted harmonic maps and the self-duality equations. Proc. London Math. Soc. (3). 55 (1987), no. 1, 127-131.
\bibitem[GZ]{GZ2018}H. G and S. Zhang, Liouville-type theorems on the complete gradient shrinking Ricci solitons. Differential Geom. Appl. 56 (2018), 42-53.
\bibitem[GW]{GW1981}R.E. Greene and H.H. Wu, Harmonic forms on noncompact Riemannian and K\"{a}hler manifolds. Michigan Math. J. 28 (1981), no. 1, 63-81.
\bibitem[Ha]{Ha1995}R.S. Hamilton, The formation of singularities in the Ricci flow. Surveys in differential geometry, Vol. \uppercase\expandafter{\romannumeral2} (Cambridge, MA, 1993), 7–136, Int. Press, Cambridge, MA, 1995.
\bibitem[Hi]{Hi1987}N.J. Hitchin, The self-duality equations on a Riemann surface. Proc. London Math. Soc. (3). 55 (1987), no. 1, 59-126.
\bibitem[Li1]{Li1990}P. Li, On the structure of complete K\"{a}hler manifolds with nonnegative curvature near infinity. Invent. Math. 99 (1990), no. 3, 579-600.
\bibitem[Li2]{Li2006}P. Li, Harmonic functions and applications to complete manifolds. \uppercase\expandafter{\romannumeral14} Escola de Geometria Diferencial. [\uppercase\expandafter{\romannumeral14} School of Differential Geometry] Instituto de Matem\'{a}tica Pura e Aplicada (IMPA), Rio de Janeiro, 2006. ii+230 pp. ISBN: 85-244-0249-0.
\bibitem[LS]{LS1984}P. Li, Peter and R. Schoen, $L^p$ and mean value properties of subharmonic functions on Riemannian manifolds. Acta Math. 153 (1984), no. 3-4, 279-301.
\bibitem[MO]{MO2024}W. Ma and J. Ou, Liouville theorem on Ricci shrinkers with constant scalar curvature and its application. J. Reine Angew. Math. 810 (2024), 283-299. 
\bibitem[MS]{MS2013}O. Munteanu and N. Se\v{s}um, On gradient Ricci solitons. J. Geom. Anal. 23 (2013), no. 2, 539-561.
\bibitem[MW1]{MW2011}O. Munteanu and J. Wang, Smooth metric measure spaces with non-negative curvature. Comm. Anal. Geom. 19 (2011), no. 3, 451-486. 
\bibitem[MW2]{MW2014}O. Munteanu and J. Wang, Holomorphic functions on K\"{a}hler-Ricci solitons. J. Lond. Math. Soc. (2) 89 (2014), no. 3, 817-831.
\bibitem[Na]{Na2010}A. Naber, Noncompact shrinking four solitons with nonnegative curvature. J. Reine Angew. Math. 645 (2010), 125-153. 
\bibitem[PRS]{PRS2011}S. Pigola, M. Rimoldi and A.G. Setti. Remarks on non-compact gradient Ricci solitons. Math. Z. 268 (2011), no. 3-4, 777-790. 
\bibitem[Sim1]{Si1988}C.T. Simpson, Constructing variations of Hodge structure using Yang-Mills theory and applications to uniformization. J. Amer. Math. Soc. 1 (1988), no. 4, 867-918.
\bibitem[Sim2]{Si1992}C.T. Simpson, Higgs bundles and local systems. Inst. Hautes \'{E}tudes Sci. Publ. Math. No. 75 (1992), 5-95.
\bibitem[UY]{UY1986}K.K. Uhlenbeck and S.T. Yau, On the existence of Hermitian-Yang-Mills connections in stable vector bundles. Frontiers of the mathematical sciences: 1985 (New York, 1985). Comm. Pure Appl. Math. 39 (1986), no. S, Suppl., S257-S293.
\bibitem[WWa]{WW2009} G. Wei and W. Wylie, Comparison geometry for the Bakry-Emery Ricci tensor. J. Differential Geom. 83 (2009), no. 2, 377–405. 
\bibitem[WZ1]{WZ2023a}D. Wu and X. Zhang, Poisson metrics and Higgs bundles over noncompact K\"{a}hler manifolds. Calc Var Partial Differential Equations. 62 (2023), no. 1, Paper No. 20, 29 pp.
\bibitem[WZ2]{WZ2023b}D. Wu and X. Zhang, Harmonic metrics and semi-simpleness. arXiv: 2304.10800, 2023.
\bibitem[WWb]{WW2015}J-Y. Wu and P. Wu, Heat kernel on smooth metric measure spaces with nonnegative curvature. Math. Ann. 362 (2015), no. 3-4, 717-742.
\bibitem[Ya1]{Ya1975}S-T. Yau, Harmonic functions on complete Riemannian manifolds. Comm. Pure Appl. Math. 28 (1975), 201-228.
\bibitem[Ya2]{Ya1976}S-T. Yau, Some function-theoretic properties of complete Riemannian manifold and their applications to geometry. Indiana Univ. Math. J. 25 (1976), no. 7, 659-670.










\end{thebibliography}
\end{document}